\documentclass[12pt]{elsarticle}
\usepackage[utf8]{inputenc}
\usepackage{amsfonts,amssymb,amsthm,amsmath}
\usepackage{graphicx}
\usepackage{xcolor}
\usepackage{hyperref}
\usepackage{array}
\usepackage{mathtools}
\usepackage{enumitem}
\numberwithin{equation}{section}
 
\theoremstyle{plain}

\newtheorem{theorem}{Theorem}[section]

\newtheorem{lemma}{Lemma}[section]

\theoremstyle{definition}

\theoremstyle{remark}
\newtheorem{remark}{Remark}[section]
\newtheorem{example}{Example}[section]

\date{}
\begin{document}
\begin{frontmatter}
 \title{On the dynamics of Stirling's iterative root-finding method for rational functions}
 \author[1]{Nitai Mandal} 
\ead{nitaimandal338@gmail.com}
\author[1]{Gorachand Chakraborty\corref{cor2}}
\ead{gorachand.chakraborty@skbu.ac.in}

\address[1]{Department of Mathematics, Sidho-Kanho-Birsha University, Purulia, India, 723104}

\cortext[cor2]{Corresponding Author} 

\begin{abstract} 
 We study the dynamics of Stirling's iterative root-finding method $St_f(z)$ for rational and polynomial functions. It is seen that the Scaling theorem is not satisfied by Stirling's iterative root-finding method. We prove that for a rational function $R(z)$ with simple zeroes, the zeroes are the superattracting fixed points of $St_{R}(z)$ and all the extraneous fixed points of $St_{R}(z)$ are rationally indifferent. For a polynomial $p(z)$ with simple zeroes, we show that the Julia set of $St_p(z)$ is connected. Also, the symmetry of the dynamical plane and free critical orbits of Stirling's iterative method for quadratic unicritical polynomials are discussed. The dynamics of this root-finding method applied to M\"{o}bius map is investigated here. We have shown that the possible number of Herman rings of this method for M\"{o}bius map is at most $2$.
\end{abstract}

\end{frontmatter}

\section{Introduction}
Let $f:\widehat{\mathbb{C}}\to \widehat{\mathbb{C}}$ be a rational function. Then, the function $f$ is of the form $f(z)=\frac{P(z)}{Q(z)}$, where $P(z)$ and $Q(z)$ are polynomials and they are co-prime to each other with $Q(z)\not\equiv 0$. The term co-prime means $P(z)$ and $Q(z)$ have no common factor. The degree of $f(z)$ is defined by $deg(f)=$max$\{deg(P), deg(Q)\}$. In complex dynamics, we mainly discuss the dynamical behaviour of the sequence of functions $\{f^n(z)\}$. The sequence of functions $\{f^n(z)\}$ is the iterative sequence defined by $f^n(z)=(f^{n-1}\circ f)(z)$, where $f^{0}$ is the identity function. For a point $z\in\widehat{\mathbb{C}}$, we define its orbit as the set $O(z)=\{f, f(z), f^2(z), \dots, f^m(z), \dots \}$. The concept of normal family is crucial to study the dynamics of rational functions.\\
 A family $F$ of  functions defined on an open set $U$ is normal on $U$ if every sequence of functions in $F$ contains a subsequence which converges uniformly on every compact subset of U.  
Here, the Riemann sphere $\widehat{\mathbb{C}}$ is divided into two complementary subsets called the Fatou set and the Julia set. The Fatou set is defined by $F(f)=\{z\in \widehat{\mathbb{C}}: \{f^n\}$ is normal in some neighbourhood of $z$ for all $n\in\mathbb{N}\}$. The complement to it is the Julia set. The Julia set is denoted by $J(f)$. The union of these two sets is called the dynamical plane $\mathbb{C}$. To characterize the dynamics of functions, periodic points help us a lot. A point $z_0$ is said to be a periodic point for $f$ of period $n$  if $f^n(z_0)=z_0$, for some $n\in\mathbb{N}$, and this $n$ is the least natural number. In particular, for $n=1$, the point $z_0$ is said to be a fixed point of $f(z)$. For a finite value $z_0$, $(f^n)^\prime(z_0)$ is said to be the multiplier of $z_0$. If $z_0=\infty$, then the multiplier is defined as $(g^n)^\prime(0)$, where $g(z)=\frac{1}{f(\frac{1}{z})}.$ A periodic point $z_0$ is called attracting, indifferent or repelling according to $\lvert (f^n)^\prime(z_0)\rvert<1, =1$ or $>1$ respectively. In the case of $\lvert (f^n)^\prime(z_0)\rvert=0$, the periodic point $z_0$ is said to be superattracting. The multiplier of an indifferent periodic point is of the form $e^{2\pi i\alpha}$, where $0\leq\alpha<1$. Then, the periodic point $z_0$ is said to be rationally indifferent if $\alpha$ is rational and irrationally indifferent otherwise. It is well known that attracting periodic points lie in the Fatou set, while rationally indifferent and repelling periodic points lie in the Julia set. A weakly repelling fixed point is a fixed point that is either repelling or rationally indifferent with multiplier $1$. Critical and free critical points are also important to find the dynamics of a function. A point $c$ is said to be a critical point of $f$ if $f^\prime(c)=0$. If a critical point is not a zero of $f$, then that critical point is called free critical point. For more informations, one can see \cite{beardon2000iteration, milnor2011dynamics, chakra2016baker}.\\
The Fatou set consists of Fatou components, which are maximal connected open subsets of the Fatou set. If we assume $U$ is a Fatou component of $f$, then $f^n(U)$ is contained in a component of the Fatou set denoted by $U_n$. There are three possible Fatou components. A component $U$ is said to be preperiodic if there exists $n>m\geq 0$ such that $U_n=U_m$. Particularly, for $m=0$ and $n\geq1$, the component $U$ is said to be periodic. A component $U$ which is not preperiodic is called a wandering component. There are four types of possible  periodic Fatou components for a rational function, namely attracting domain, parabolic domain, Siegel disc and Herman ring. Herman ring is a doubly connected periodic Fatou component. Let $U$ be a $p$ periodic Herman ring of $f$ then there exists an analytic homeomorphism $\phi:U\to A$, where $A=\{z: 1<\lvert z \rvert<r\}$, $r>1$ is an annulus such that $(\phi\circ f^p\circ\phi^{-1})(z)=e^{2\pi i\alpha}z$ for some $\alpha\in\mathbb{R}\setminus \mathbb{Q}$. For details about the Herman ring, one can see \cite{chakra2018herman, chakraborty2021configurations}.\\ 
If $f(z)$ is analytic at $z_0,$ then it's Taylor's series expansion about $z_0$ is of the form $a_n(z-z_0)^n+a_{n+1}(z-z_0)^{n+1}+\dots
$ for some $n\geq 0$, where $a_n\neq 0$. Here, $n$ is called the local degree of $f$ at $z_0$ and we denote it by $deg(f, z_0)$. If $f(\infty)=\infty$, then the local degree of $f$ at $\infty$ is defined by $deg(\frac{1}{f(\frac{1}{z})}, 0)$. If $f:U\to V$ is a function such that the inverse image of every compact set in $V$ is compact in $U$ then $f$ is called proper map. Connectivity of a set $A$ is the number of maximally connected open subsets of $\widehat{\mathbb{C}}\setminus A$. Now, we are giving Riemann-Hurwitz formula {[p-85, \cite{beardon2000iteration}]}.
\begin{lemma}[Riemann-Hurwitz formula]
 If $f:U_1\to U_2$ is a rational function between two Fatou components $U_1$ and $U_2$, then it is a proper map of some degree $d$ and $c(U_1)-2=d(c(U_2)-2)+CP$, where $c(.)$ denotes the connectivity of a domain and $CP$ is the number of critical points of $f$ in $U_1$ counting their multiplicity.   
 \end{lemma}
 
The iteration of root-finding algorithms and their dynamical behaviour is an active area of research.  Root-finding algorithms are used to find the approximate roots of functions. We say that a map $f\to T_f$ carrying a complex-valued function $f(z)$ to a function $T_f:\mathbb{\widehat{C}}\to \mathbb{\widehat{C}}$ is an iterative root-finding algorithm if $T_f(z)$ has a fixed point at every root of $f(z)$, and given an initial guess $z_0$, the sequence of iterates $\{z_k\}_{k\geq 0}$, where $z_{k+1}=T_f(z_k)$, converges to a root $r$ of $f(z)$ whenever $z_0$ is sufficiently close to $r$ \cite{amat2004review}.
We denote $T_f$ as a root-finding method for a function $f$. There are several root-finding methods available for finding the roots. For an attracting fixed point $z_0$, the basin of attraction of $z_0$ is defined by $B(z_0)=\{z\in\widehat{\mathbb{C}}: T_f^n(z)\to z_0$ as $n\to \infty$\}. This is an open set, but not necessarily connected. The component of $B(z_0)$ containing the fixed point $z_0$ is said to be the immediate basin of attraction of $z_0$. Fixed points of the root-finding method $T_f(z)$ may or may not be the roots of $f$. The fixed points of $T_f(z)$ which are not roots of $f(z)$, are called the extraneous fixed points.\\
We denote Stirling's iterative root-finding method as $St_f(z)$, applied to the function $f(z)$ throughout this paper. Stirling's iterative method for a function $f$ is given by $$St_f(z)=z-\frac{f(z)}{f^\prime(z-f(z))}.$$ It is a second-order iterative root-finding method \cite{amat2004review}. The work on the dynamics of Stirling's iterative method is less available in literature. Most of the dynamical characteristics remain unrecognized. Our main motivation of this paper is to explore the Stirling's iterative method dynamically. Expanding and extending the applicability of Stirling's iterative method are shown in \cite{argyros2018expanding, amoros2019extending}. Here, authors investigate semilocal and local convergence of Stirling’s iterative method. This method is used to find the fixed points of nonlinear operator equation. One of the most popular and useful root-finding methods is Newton's method. The dynamics of quadratic and cubic Newton's method of rational functions are described in \cite{nayak2022quadratic}. The Scaling theorem plays an important role for characterizing the dynamics of a root-finding method. If a root-finding method satisfies the Scaling theorem, then for different polynomials we can have the same root-finding method up to affine conjugacy. As a consequence, the root-finding method of two different polynomials may have similar dynamics. Some well-known root-finding methods, such as Newton's method \cite{wang2007julia},  Halley's method \cite{kneisl2001julia}, K\"{o}nig's family of root-finding methods \cite{buff2003konig}, Chebyshev's method \cite{nayak2022julia}, etc. satisfy the Scaling theorem. There are also some root-finding methods that do not satisfy the Scaling theorem. Stirling's and Steffensen’s iterative root-finding methods are examples of such type. Therefore, study on the dynamics of Stirling's iterative root-finding method is worth doing.\\
In Section \ref{sec2}, we prove that Stirling's iterative method does not satisfy the Scaling theorem. Then, we show that for a rational function $R(z)=\frac{P(z)}{Q(z)}$, with simple zeroes, all the zeroes of $R(z)$ are the superattracting fixed points of $St_R(z)$ and the finite extraneous fixed points of $St_R(z)$ satisfy $Q(z-R(z))=0$ and are rationally indifferent. The connectivity of the Julia set of $St_p(z)$ is discussed for a polynomial $p(z)$ with simple zeroes. In Section \ref{sec3}, the symmetry of the dynamical plane and critical orbits of Stirling's iterative method for quadratic unicritical polynomials are discussed. In Section \ref{sec4}, we investigate the dynamics of Stirling's iterative method for the M\"{o}bius map $M(z)$. We prove that the Julia set of $St_M(z)$ is disconnected for $a=0$. Also, we have found that if Herman rings exist for this method, then the possible number of Herman rings is at most $2$. In this paper, we give some images generated by Python programming and analyze them from a dynamical point of view. Finally, we provide a comparison table between the dynamical properties of Newton's method and Stirling's iterative method.

\section{Stirling's iterative method for rational functions}\label{sec2}
A root-finding method $T_f$ is said to satisfy the Scaling theorem if for any polynomial $f$, any non-zero complex number $\lambda$ and every affine map $T$, $(T\circ T_g\circ T^{-1})(z)=T_f(z)$ where $g(z)=\lambda(f\circ T)(z)$.
The result that Stirling’s iterative method does not satisfy the Scaling theorem has already been stated in [p-19, \cite{amat2004review}], but proof has not been given. We are giving  the complete proof of the said fact in the below lemma.
\begin{lemma}
Stirling's iterative method $St_f(z)$ does not satisfy the Scaling theorem.
\end{lemma}
\begin{proof}
 Let $T(z)=az+b$, $a\neq 0$ be an affine map. Also, let us consider $f(z)$ and $g(z)=\lambda(f\circ T)(z)$ be two polynomials, where $\lambda\neq 0\in \mathbb{C}$. Now, $(T\circ St_g\circ T^{-1})(z)=T\circ \Big(T^{-1}(z)-\frac{g(T^{-1}(z))}{g^{\prime}\big(T^{-1}(z)-g(T^{-1}(z))\big)}\Big)$. By using $T^{-1}(z)=\frac{z-b}{a}, g(T^{-1}(z))=\lambda f(z)$ and $g^\prime\big(\frac{z-b}{a}-\lambda f(z)\big)=a\lambda f^\prime\big(z-a\lambda f(z)\big)$, we have $(T\circ St_g\circ T^{-1})(z)=z-\frac{f(z)}{f^\prime(z-a\lambda f(z))}$. Hence, for all the values of $\lambda$ $(\lambda\neq\frac{1}{a})$, we see that $(T\circ St_g\circ T^{-1})(z)\neq z-\frac{f(z)}{f^\prime(z-f(z))}=St_f(z)$ for all $z\in\mathbb{C}$.     
\end{proof}
An important observation is that Stirling's iterative root-finding method can be applied to any non-constant polynomials having repeated zeroes. For example, let $f$ be a polynomial with a zero at $\alpha$ of multiplicity $n\geq 2$. Then $f(z)=(z-\alpha)^ng(z)$, where $g(z)$ is a polynomial with $g(\alpha)\neq 0$. Then $f^\prime(z)=(z-\alpha)^{n-1}h(z)$, where $h(z)=ng(z)+(z-\alpha)g^\prime(z)$ and $h(\alpha)=ng(\alpha)\neq 0$. Also, $f^\prime(z-f(z))=(z-\alpha)^{n-1}\phi(z)$, where $\phi(z)=\{1-(z-\alpha)^{n-1}g(z)\}^{n-1}h(z-f(z))$ and $\phi(\alpha)\neq 0$. Therefore, $St_f(z)=z-(z-\alpha)\frac{g(z)}{\phi(z)}$. This shows that $\alpha$ is a fixed point of $St_{f}(z)$. Thus the dynamics of Stirling's iterative method applied to polynomials with repeated zeroes may be studied in the same line as in this paper we have studied for simple zeroes.\\
Now, we discuss about the dynamics of Stirling's iterative method applied to rational functions with simple zeroes.
\begin{theorem}\label{th21}
   Let $R(z)=\frac{P(z)}{Q(z)}$ be a rational function, where $P(z)$ is a polynomial with simple zeroes and $Q(z)\not\equiv 0$ is a polynomial. Then, all the zeroes of $R(z)$ are the superattracting fixed points of $St_R(z)$ as well as all the finite extraneous fixed points of $St_R(z)$ satisfy $Q(z-R(z))=0$ and are rationally indifferent.
\end{theorem}
\begin{proof}
  Let us assume that $P(z)=\Sigma_{i=0}^{n}a_iz^i$, $a_i\in\mathbb{C}$ and $Q(z)=\Sigma_{j=0}^{m}b_jz^j$, $b_j\in\mathbb{C}$ are two polynomials of degree $n$ and $m$ respectively. Then, $R^\prime(z-R(z))=\frac{Q(z-R(z))P^\prime(z-R(z))-P(z-R(z))Q^\prime(z-R(z))}{Q(z-R(z))^2}$. Therefore, the Stirling's method for a rational function $R(z)$ is given by $St_R(z)=z-\frac{P(z)Q(k(z))^2}{Q(z)\big(Q(k(z))P^\prime(k(z))-P(k(z))Q^\prime(k(z))\big)}$, where $k(z)=z-R(z)$. Now, $St_R(z)=z$ gives the fixed points. This shows the solutions of $P(z)=0$ or $Q(k(z))=0$ are the fixed points of $St_R(z)$. To find the nature of these fixed points, we have to find out the multiplier value of $St_R(z)$. Let us write $S(z)=Q(k(z))P^\prime(k(z))-P(k(z))Q^\prime(k(z))$ then
\begin{equation}\label{eq1}
 St^\prime_R(z)= 1-\dfrac{\splitdfrac{Q(z)S(z)\big(P^\prime(z)Q(k(z))^2+2P(z)Q(k(z))(1-R^\prime(z))\big)}{-P(z)Q(k(z))^2\frac{d}{dz}(Q(z)S(z))}}{\big(Q(z)S(z)\big)^2}.
 \end{equation}
 If $z_0$ is a simple zero of $R(z)$ then $R(z_0)=0$ and $P(z_0)=0$, but $P^\prime(z_0)\neq 0$. So, $St^\prime_R(z_0)=1-\frac{P^\prime(z_0)Q(z_0)^2}{Q(z_0)S(z_0)}$. By using $S(z_0)=Q(z_0)P^\prime(z_0)$, we get $St^\prime_R(z_0)=0$. Hence, $z_0$ is a superattracting fixed point of $St_R(z)$. If $z_e$ is a fixed point of $St_R(z)$ such that $Q(k(z_e))=0$, then $Q(z_e-R(z_e))=0$. Let us assume that $P(z_e)=0$, it implies that $Q(z_e)=0$, which is a contradiction, as $R(z)=\frac{P(z)}{Q(z)}$ is a rational function. Thus, $P(z_e)\neq 0$. Hence, the fixed point $z_e$ is an extraneous fixed point of $St_R(z)$. Since $Q(k(z_e))=0$, from Equation \ref{eq1}, it is easy to see $St^\prime_R(z_e)=1$. This shows that the extraneous fixed point $z_e$ is rationally indifferent. This completes the proof.
\end{proof}
\begin{remark}\label{rem 1}
 If $Q(z)=1$, then $R(z)=P(z)$, which becomes a polynomial function. Since, from the above Theorem \ref{th21}, we can see that all the finite extraneous fixed points are the solutions of $Q(z-R(z))=0$. Then, Stirling's iterative method has no finite extraneous fixed point. So, all the finite fixed points of $St_P(z)$ are superattracting. 
\end{remark}
Now, we give an example to understand the above Theorem \ref{th21} geometrically.
\begin{example}
   Let us consider a function $f(z)=\frac{2z-1}{z}$, which has a simple zero at $z=\frac{1}{2}$. For this function $f$, Stirling's iterative method becomes $St_f(z)=z-\frac{(z-1)^4(2z-1)}{z^3}$. So, $z=\frac{1}{2}$ is a superattracting fixed point of $St_f(z)$, and $z=1$ is an extraneous fixed point that is rationally indifferent. Corresponding to these superattracting and rationally indifferent fixed points, the Fatou set contains a superattracting domain and a parabolic domain, respectively, which is given below in Figure \ref{fig:0}.
\end{example}
\begin{figure}[h!]
    \centering
    \includegraphics[width=9.5cm, height = 6.5cm]{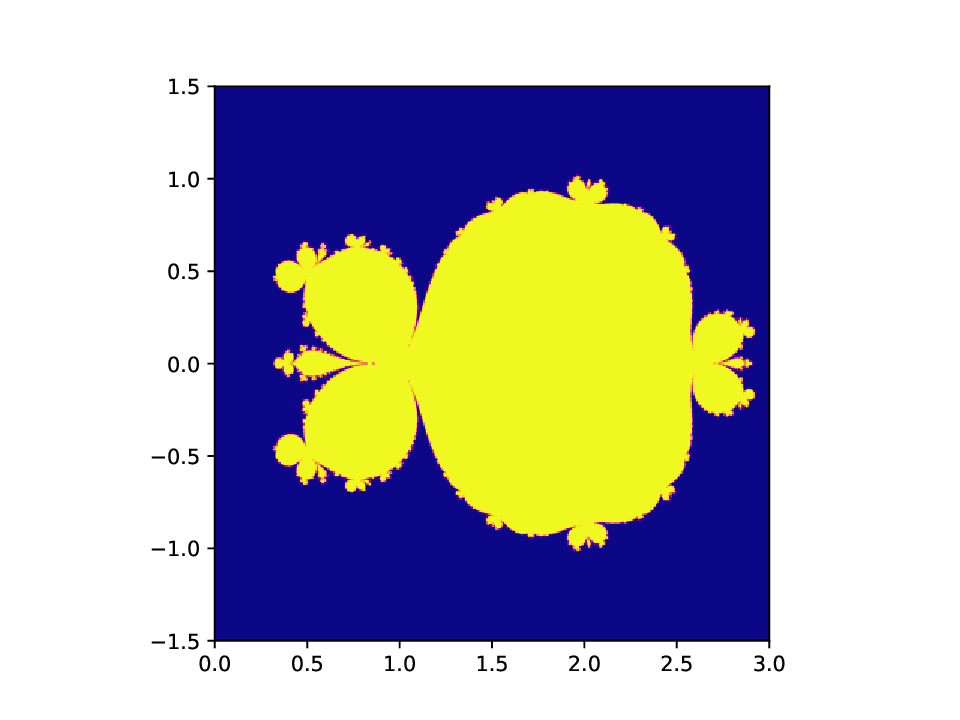}
    \caption{The dynamical plane of $St_f(z)$ for $f(z)=\frac{2z-1}{z}$}
    \label{fig:0}
\end{figure}
Our next result discusses the connectivity of the Julia set of Stirling's iterative method for any polynomials with simple zeroes.
\begin{theorem}\label{th22}
Let $P(z)$ be a polynomial with simple zeroes. Then\\
   $i)$ the point $\infty$ is a fixed point and lies in the Julia set of $St_P(z)$.\\
   $ii)$ the Julia set of $St_P(z)$ is connected.  
   \end{theorem}
\begin{proof}
    $i)$ Let $P(z)=a_0+a_1z+a_2z^2+\dots+a_nz^n$, $a_i\in\mathbb{C}$ and $a_n\neq 0$  be a polynomial of degree $n$ that has simple zeroes. Then, Stirling's iterative method for this polynomial $P(z)$ is given by $St_P(z)=z-\frac{P(z)}{P^\prime(z-P(z))}$. Now, $P(\frac{1}{z})=\frac{a_n+a_{n-1}z+\dots+a_0z^n}{z^n}$ and so, $\frac{1}{z}-P(\frac{1}{z})=\frac{K(z)}{z^n}$, where $K(z)=z^{n-1}-(a_n+a_{n-1}z+\dots+a_0z^n)$. Therefore, $P^\prime(\frac{1}{z})=a_1+2a_2\frac{1}{z}+3a_3\frac{1}{z^2}+\dots +na_n\frac{1}{z^{n-1}}$. Now, $P^\prime(\frac{1}{z}-P(\frac{1}{z}))=a_1+2a_2\frac{K(z)}{z^n}+3a_3\frac{K(z)^2}{z^{2n}}+\dots +na_n \frac{K(z)^{n-1}}{z^{n(n-1)}}=\frac{1}{z^{n(n-1)}}S(z)$, where $S(z)=a_1z^{n(n-1)}+2a_2K(z)z^{n^2-2n}+\dots +na_nK(z)^{n-1}$. We write $St_P(\frac{1}{z})=\frac{S(z)-z^{n^2-2n+1}(z^{n-1}-K(z))}{zS(z)}$  as $H(z)=\frac{1}{St_P(\frac{1}{z})}=\frac{zS(z)}{S(z)-z^{n^2-2n+1}(z^{n-1}-K(z))}$. Therefore, $H^\prime(0)=1$. This shows that the point $\infty$ is a rationally indifferent fixed point that lies in the Julia set of $St_P(z).$\\
 $ii)$ From the above Remark \ref{rem 1}, we see that $St_P(z)$ has no finite extraneous fixed point. All the finite fixed points of $St_P(z)$ are superattracting. Also, the fixed point $\infty$ is a rationally indifferent fixed point. In [Corollary II, \cite{shishikura1990connectivity}], we know that if a rational function has only one fixed point, which is weakly repelling fixed point, then its Julia set is connected. Hence, the Julia set of $St_P(z)$ is connected. In another way, we can say every Fatou component of $St_P(z)$ is simply connected, which means the Fatou set of $St_P(z)$  does not contain any Herman rings.  
\end{proof}
\begin{remark}
  From the above Theorem \ref{th22} and Remark \ref{rem 1}, we see that  all the fixed points of $St_P(z)$ are either superattracting or rationally indifferent. Thus, we conclude here that the Fatou set of $St_P(z)$ has no invariant Siegel disc. Consequently, any invariant Fatou component of $St_P(z)$ must be either attracting domain or parabolic domain.   
\end{remark}
\section{ Stirling's iterative method for unicritical polynomials}\label{sec3}
 A general quadratic unicritical polynomial is defined by $f_\lambda(z)=\lambda(z-\alpha)^2+\beta$, where $\lambda, \beta\in\mathbb{C}\setminus\{0\}$ and  $\alpha\in\mathbb{C}$. For this function $f_\lambda(z)$, Stirling's iterative method is $St_{f_\lambda}(z)=z+\frac{\lambda(z-\alpha)^2+\beta}{2\lambda\big(\lambda z^2-(2\alpha\lambda+1)z+(\alpha^2\lambda+\beta+\alpha)\big)}$.\\ In particular, if $\lambda=1, \alpha=0$ and $\beta\in\mathbb{R}\setminus\{0\}$ then the Stirling's iterative method is $St_{f_1}(z)=z+\frac{z^2+\beta}{2(z^2-z+\beta)}$. The fixed points of $St_{f_1}(z)$ are $z=\pm \sqrt{-\beta}$. Also, the critical points are $z=\pm \sqrt{-\beta}$ and the free critical points are $z=\frac{2\pm\sqrt{2(1-2\beta)}}{2}$. From Theorem \ref{th22}, the point $\infty$ is a rationally indifferent fixed point of $St_{f_1}(z)$. The free critical points of $St_{f_1}(z)$ for $f_1(z)=z^2+\beta$, depend on the parameter $\beta$.\\
 Now, we give two results regarding the symmetry of the dynamical plane and orbits of the free critical points with respect to the $x$-axis.
\begin{theorem}\label{th3.1}
The dynamical plane of $St_{f_1}(z)$ is symmetric with respect to the $x$-axis. 
\end{theorem} 
\begin{proof}
We have, $St_{f_1}(z)=\frac{2z^3-z^2+2\beta z+\beta}{2(z^2-z+\beta)}$ $(z\neq\frac{1\pm\sqrt{1-4\beta}}{2})$, where $\beta\in\mathbb{R}\setminus\{{0}\}$. It is a rational function with real coefficients. Clearly, $St_{f_1}(\overline{z})=\overline{St_{f_1}(z)}$. Now, $St^2_{f_1}(\overline{z})=St_{f_1}(St_{f_1}(\overline{z}))=St_{f_1}(\overline{St_{f_1}(z)})=\overline{St_{f_1}(St_{f_1}(z)})=\overline{St^2_{f_1}(z)}$. Let us assume $St^m_{f_1}(\overline{z})=\overline{St^m_{f_1}(z)}$ holds for $n=m$. Now, $St^{m+1}_{f_1}(\overline{z})=St_{f_1}(St^m_{f_1}(\overline{z}))=St_{f_1}(\overline{St^m_{f_1}(z)})=\overline{St_{f_1}(St^m_{f_1}(z))}=\overline{St^{m+1}_{f_1}(z)}$ holds for $n=m+1$. So, by Principle of Mathematical Induction $St^n_{f_1}(\overline{z})=\overline{St^n_{f_1}(z)}$ holds for all $n\in\mathbb{N}$. Let $z\in F(St_{f_1}(z))$. Then, there is a Fatou component $U$ of $z$ such that the family ${St_{f_1}(z)}$ is well-defined and normal in $U$. We know that $J(St_{f_1}(z))$ is an uncountable set. So, there are at least two complex values $c_1, c_2\not\in St_{f_1}(U)$ for any $n\in\mathbb{N}$. Let $V$ be the mirror image of $U$ with respect to the $x$-axis. Thus, $V$ is a neighbourhood of the point $\overline{z}$. Let $\overline{c_i}\in St^n_{f_1}(V)$, for some $n\in\mathbb{N}$ and $i=\{1, 2\}$. Also, let $\overline{c_i}=St^k_{f_1}(c)$, for some $c\in V$ and $k\in\mathbb{N}$. Now, $c_i=\overline{\overline{c_i}}=\overline{St^k_{f_1}(c)}=St^k_{f_1}(\overline{c})$, where $\overline{c}\in U$. This is a contradiction. Hence, $\overline{z}$ is a neighbourhood of $V$ such that $\overline{c_i}\not\in St^n_{f_1}(V)$, for $i=\{1, 2\}$. This shows that the family $\{St_{f_1}(V)\}$ omits the same two values. So, by Montel's criteria of normality, the family $\{St^n_{f_1}\}$ is normal in $V$, and so $\overline{z}\in F(St_{f_1}(z))$. Similarly, we can show that if $\overline{z}\in$ $F(St_{f_1}(z))$, then $z\in F(St_{f_1}(z))$. So, $z\in F(St_{f_1}(z))$ if and only if $\overline{z} \in F(St_{f_1}(z))$. Consequently, $z\in J(St_{f_1}(z))$ if and only if $\overline{z} \in J(St_{f_1}(z))$. This shows that the dynamical plane is symmetric with respect to the $x$-axis. 
\end{proof}
\begin{remark}
The Stirling's iterative method applied to any quadratic unicritical polynomial with real coefficients gives a rational function with real coefficients. Since the proof of Theorem \ref{th3.1} relies on the fact that the Stirling's iterative method should give a function with real coefficients, so the Theorem \ref{th3.1} may be generalized for any real $\lambda$.  
\end{remark}
\begin{theorem}
The orbits of the free critical points of $St_{f_1}(z)$ are symmetric with respect to the $x$-axis.   
\end{theorem}
\begin{proof}
 For the free critical point $z_0=\frac{2-\sqrt{2(1-2\beta)}}{2}$, we have $St_{f_1}(z_0)=g(\beta)$. So, $g(\beta)=\frac{7-4\beta-5\sqrt{2(1-2\beta)}}{2\big(1-\sqrt{2(1-2\beta)}\big)}$ $(\beta\neq\frac{1}{4})$ is a function with respect to the parameter $\beta$. It is easy to check $g(\overline{\beta})=\overline {g(\beta)}$ and $g^2(\overline{\beta})=g(g(\overline{\beta}))=g(\overline{g(\beta)})=\overline{g(g(\beta))}=\overline{g^2(\beta)}$. Let us assume $g^m(\overline{\beta})=\overline{g^m(\beta)}$ holds for $m\in\mathbb{N}$. Now, $g^{m+1}(\overline{\beta})=g(g^m(\overline{\beta}))=g(\overline{g^m(\beta)})=\overline{g(g^m(\beta)}=\overline{g^{m+1}(\beta)}$ holds for $n=m+1$. So, by Principle of Mathematical Induction, $g^n(\overline{\beta})=\overline{g^n(\beta)}$, holds for all $n\in\mathbb{N}$. We proceed in a similar way for the critical point $z_1=\frac{2+\sqrt{2(1-2\beta)}}{2}$. Hence, the orbits of the free critical points are symmetric with respect to the $x$-axis.  
\end{proof}
\begin{remark}
The Fatou set of $St_{f_1}(z)$ contains two invariant immediate superattracting basin. In Figure \ref{fig2}(a), we take $\beta=-4$. Then, the superattracting fixed points of $St_{f_1}(z)$ are $\pm 2$. Let $A_2$ and $A_{-2}$ are two immediate superattracting basins of attraction. The large yellow region is the basin of $A_2$ and the small yellow region is the basin of $A_{-2}$. In Figure \ref{fig2}(b), we take $\beta=4$, and then two immediate superattracting basins are $A_{2i}$ and $A_{-2i}$.
\end{remark}
\begin{figure}[h!]
    \centering
    \includegraphics[width=1\textwidth]{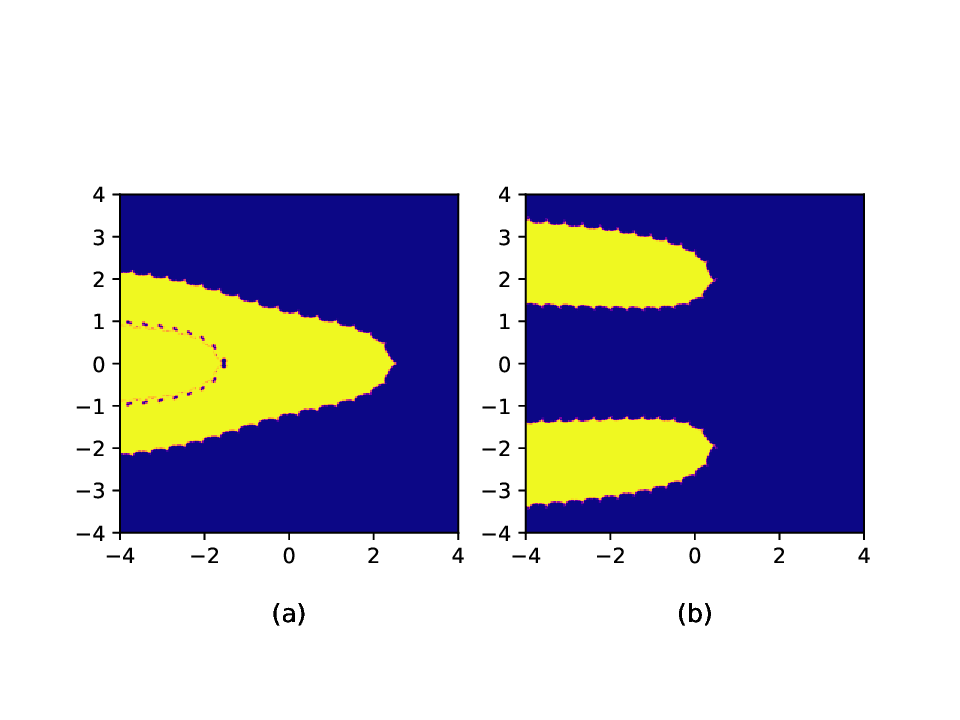}
    \caption{The Julia set of $St_{f_1}(z)$ for $\beta=-4$ and $\beta=4$}
    \label{fig2}
\end{figure}
\begin{remark}
The degree of $St_{f_1}(z)$ is $3$. So, from \cite{beardon2000iteration}, the number of fixed points as well as number of critical points of $St_{f_1}(z)$ is $4$. Among them, the two critical points namely $z=\sqrt{-\beta}$ and $z=-\sqrt{-\beta}$ belong to the immediate basins of attraction $A_{\sqrt{-\beta}}$ and  $A_{-{\sqrt{-\beta}}}$ respectively.
\end{remark}
\section{ Stirling's iterative method applied to M\"{o}bius map}\label{sec4}
M\"{o}bius map is a rational function of degree one. This map can be written in the following way $M(z)=\frac{az+b}{cz+d}$, where $a, b, c, d\in\mathbb{C}$ and $ad-bc\neq 0$. We take $c\neq 0$ throughout this section because if $c=0$ then M\"{o}bius map become linear polynomial and the case become trivial. Also, here both $b$ and $d$ can not be zero as $ad-bc\neq 0$. Now, two cases will occur for the value of $a$.
\textbf{Case I: when $a=0$ with $b\neq 0$ and $d\in\mathbb{C}$}\\
For $a=0$, M\"{o}bius maps become $M(z)=\frac{b}{cz+d}$. So, clearly, $M(z)$ is a rational function of degree one that has no zero. Now, we have a theorem regarding the dynamical behaviour of $St_M(z)$, which is given below.
\begin{theorem}
The following results are true for Stirling's iterative method $St_M(z)$ where $M(z)=\frac{b}{cz+d}$,\\
$i)$ $deg(St_M(z))=4$, and the number of critical points is $6$.\\
$ii)$ the point $\infty$ is an attracting fixed point.\\
$iii)$ If Herman ring exists, then the number of Herman rings will be at most $1$.\\
$iv)$ the Julia set is disconnected.
\end{theorem}
\begin{proof}
$i)$ Given, $M(z)=\frac{b}{cz+d}$, then $M^\prime(z-M(z))=\frac{-bc(cz+d)^2}{\big(c^2z^2+2cdz+(d^2-bc)\big)^2}$. Stirling's iterative method is $St_M(z)=z+\frac{\big((cz+d)^2-bc\big)^2}{c(cz+d)^3}=\frac{cz(cz+d)^3+\big((cz+d)^2-bc\big)^2}{c(cz+d)^3}$. Since, $c\neq 0$, this shows the degree of $St_M(z)$ is $4$. From [Corollary 2.7.2, \cite{beardon2000iteration}], we know that for a rational function of degree $d$ has $2d-2$ critical points. Hence, the number of critical points of $St_M(z)$ is $8-2=6$.\\
$ii)$ It is easy to say that the point $\infty$ is a fixed point of $St_M(z)$. Write, $g(z)=\frac{1}{St_M(\frac{1}{z})}$ for calculating the multiplier value at $\infty$. So, $g(z)=\frac{A(z)}{B(z)}$, where $A(z)=cz(c+dz)^3$ and $B(z)=c(c+dz)^3+\big((c+dz)^2-bcz^2\big)^2$. By using $A(0)=0, A^\prime(0)=c^4, B(0)=2c^4$ and $B^\prime(0)=7c^3d$, we get $g^\prime(0)=\frac{1}{2}$. This shows that the fixed point $\infty$ is an attracting fixed point of $St_M(z)$.\\
$iii)$  We know a relation for a  rational function $f$ of degree $d$ is $n_{AB}+ n_{PB}+n_{SD}+2n_{HR}+n_{\text{Cremer}}\leq 2d-2$, where $n_{AB}, n_{PB}, n_{SD}, n_{HR}$ and $ n_{Cremer}$ denote the number of attracting domain cycles, number of parabolic basin cycles, number of Siegel disc cycles, number of Herman ring cycles, and number of Cremer cycles respectively [Corollary 2, \cite{shishikura1987quasiconformal}]. Here, $d=deg(St_M(z))=4$. From Theorem \ref{th21}, we know that all the finite extraneous fixed points are rationally indifferent. So, the Fatou set of $St_M(z)$ contains two parabolic domains. Therefore, $n_{PB}\geq 2$. Also, from (ii), we see that the fixed point $\infty$ is an attracting fixed point, hence the Fatou set of $St_M(z)$ contains an attracting domain. So, $n_{AB}\geq 1$. The above relation can be written, $n_{AB}+n_{PB}+2n_{HR}\leq 2d-2$. This gives $n_{HR}\leq\frac{3}{2}$. This shows that if Herman rings exist, then the number of Herman rings will be at most $1$.\\ 
$iv)$ Let $A_\infty$ be an attracting domain of $St_M(z)$ containing $\infty$. Now, we consider a rational function $St_M(z): A_\infty \to A_\infty$. It is known that the connectivity of an invariant Fatou component is one of the values $1, 2$ or $\infty$ [Theorem 7.5.3, \cite{beardon2000iteration}]. If we assume $c(A_\infty)$ and $CP$ be the connectivity and the number of critical points of $A_\infty$. Then, by Riemann-Hurwitz formula, we have $c(A_\infty)-2=d(c(A_\infty)-2)+CP$, where $d$ is the local degree of $St_M(z)$ at $\infty$. Now, the local degree of $St_M(z)$ at $\infty$ is denoted by $deg(St_M(z), \infty)$ and is defined by $deg(St_M(z), \infty)=deg(g(z), 0)$, where $g(z)=\frac{1}{St_M(\frac{1}{z})}$. From $ii)$, we see that $g^\prime(0)=\frac{1}{2}\neq 0$. So, the local degree of $St_M(z)$ at $\infty$ is $d=deg(St_M(z), \infty)=1$. Now, form $c(A_\infty)-2=c(A_\infty)-2+CP$, we get $CP=0$ for the finite values of $c(A_\infty)$. This is a contradiction; since $A_\infty$ is an attracting domain, it must contain at least one critical point. Hence, $A_\infty$ is infinitely connected. We know that if $R$ is a rational map, then $J(R)$ is connected if and only if every Fatou component is simply connected [Theorem 5.1.6, \cite{beardon2000iteration}]. As $A_\infty$ is infinitely connected, the Julia set of $St_M(z)$ is disconnected. This completes the proof.
\end{proof}
\begin{remark}
  For case-I, Stirling's iterative method $St_M(z)$ has no finite superattracting fixed point. All the finite fixed points are extraneous. These finite extraneous fixed points are $z_e=\frac{-d\pm\sqrt{bc}}{c}$ of multiplicity $2$.  
\end{remark}
\textbf{Case II: when $a\neq0$, and $b, d\neq 0$}\\
For $a\neq 0$, M\"{o}bius maps are of the form $M(z)=\frac{az+b}{cz+d}$. The map of this particular form has a simple zero namely $-\frac{b}{a}$. Now, we investigate the dynamics of $St_M(z)$ for this case.
\begin{theorem}\label{th32}
 The following results are true for Stirling's iterative method $St_M(z)$ where $M(z)=\frac{az+b}{cz+d}$, \\ 
 $i)$ $deg(St_M(z))=5$, and the number of critical points is $8$.\\
 $ii)$ the finite extraneous fixed points are $\frac{-(2d-a)\pm \sqrt{4(bc-ad)+a^2}}{2c}$ of multiplicity $2$.\\
 $iii)$ the point $\infty$ is a superattracting fixed point.\\
 $iv)$ If Herman rings exist, then the number of Herman rings will be at most $2$.
\end{theorem}
\begin{proof}
  $i)$ Since, $M(z)=\frac{az+b}{cz+d}$. Therefore, $M^\prime(z)=\frac{ad-bc}{(cz+d)^2}$ and so $M^\prime(z-M(z))=\frac{(ad-bc)(cz+d)^2}{(c^2z^2+2cdz-acz+d^2-bc)^2}$. The Stirling's iterative method is $St_M(z)=z-\frac{(az+b)(c^2z^2+2cdz-acz+d^2-bc)^2}{(cz+d)^3(ad-bc)}$. Clearly, the degree of $St_M(z)$ is $5$ and so the number of critical points is $10-2=8$.\\
  $ii)$ The equation $St_M(z)=z$, gives the fixed points of $St_M(z)$. So, from $St_M(z)=z-\frac{(az+b)(c^2z^2+2cdz-acz+d^2-bc)^2}{(cz+d)^3(ad-bc)}$, we see that $z=-\frac{b}{a}$ is a fixed point. Also, the finite extraneous fixed point is the solution of $(c^2z^2+2cdz-acz+d^2-bc)^2=0$. Solving this we get $z=\frac{-(2d-a)\pm \sqrt{4(bc-ad)+a^2}}{2c}$ with multiplicity $2$.\\
  $iii)$ The Stirling's iterative method is $St_M(z)=z-\frac{(az+b)(c^2z^2+2cdz-acz+d^2-bc)^2}{(cz+d)^3(ad-bc)}$. Now, write $g(z)=\frac{1}{St_M(\frac{1}{z})}$. Therefore, $g(z)$ can be written as, $g(z)=\frac{A(z)}{B(z)}$, where $A(z)=z^2(ad-bc)(c+dz)^3$ and $B(z)=z(ad-bc)(c+dz)^3-(a+bz)(c^2+2cdz-acz+(d^2-bc)z^2)^2$. By using $A(0)=A^\prime(0)=0,$ and $B(0)=-ac^4$ we have $g^\prime(0)=0$. This shows that the fixed point $\infty$ is a superattracting fixed point.\\
  $iv)$ Since, the fixed points $z=-\frac{b}{a}$ and $\infty$ are two superattracting fixed points. So, the Fatou set of $St_M(z)$ contains two superattracting domains. Therefore, $n_{AB}\geq 2$. Also, we prove that the finite extraneous fixed points are rationally indifferent, so $n_{PB}\geq 2$. We know the relation, $n_{AB}+n_{PB}+2n_{HR}\leq 2d-2$ [Corollary 2, \cite{shishikura1987quasiconformal}]. Here, $d$ is the degree of $St_M(z)$, which is $4$. Hence, $n_{HR}\leq 2$. This shows that, if Herman rings exist for this method, then the number of Herman rings will be at most $2$.
\end{proof}
\begin{remark}
From the definition of Siegel disc, we know that any invariant Siegel disc must contain an irrationally indifferent fixed point [p-126, \cite{milnor2011dynamics}]. Since, $St_M(z)$ does not have any irrationally indifferent fixed point, the Fatou set of $St_M(z)$ does not contain any invariant Siegel disc. 
\end{remark}
\begin{remark}
 Let us define a set $S=\{ T(z)=\frac{az+b}{cz+d},$ where $a, c\in\mathbb{C}\setminus\{0\}, b, d\in\mathbb{C}$ and $ad-bc=1\}$. This family $S$ is a special family of the M\"{o}bius map. For a M\"{o}bius map $T(z)\in S$, by using Theorem \ref{th21}, we have $z=-\frac{b}{a}$ is a superattracting fixed point of $St_T(z)$. Also, the finite extraneous fixed points $z_e=\frac{(a-2d)\pm\sqrt{a^2-4}}{2c}$ are rationally indifferent fixed points of $St_T(z)$. From Theorem \ref{th32}, we know that the fixed point $\infty$ is a superattracting fixed point of $St_T(z)$.    
\end{remark}
To understand the above discussion, we give an example of this method for choosing some particular values of the parameters $a$, $b$, $c$ and $d$ in the M\"{o}bius map $M(z)$.
\begin{example}
Let us assume $a=0$, $b=c=i$ and $d=1$. Then the M\"{o}bius map is $M(z)=\frac{i}{iz+1}$ and Stirling's iterative method is $St_M(z)=z-\frac{i(z^2-2iz-2)^2}{(iz+1)^3}$. From the above discussion, we say that this method has two finite extraneous fixed points $z=(i\pm1)$ and these are rationally indifferent. So, the Fatou set of $St_M(z)$ contains two parabolic domains, which is shown in Figure \ref{fig:3}(a). Also, for $a=c=d=1$, and $b=0$. The M\"{o}bius map becomes $M(z)=\frac{z}{z+1}$. Therefore, Stirling's iterative method is $St_M(z)=\frac{z^2(1-z^2-z^3)}{(z+1)^3}$. The fixed point $z=0$ is a superattracting fixed point. Also, the finite extraneous fixed points $z_e=\frac{-1\pm i\sqrt{3}}{2}$ are rationally indifferent. Thus, the Fatou set of $St_M(z)$ contains one superattracting and two parabolic domains. This is shown in Figure \ref{fig:3}(b).
\end{example}
\begin{figure}[h!]
    \centering
    \includegraphics[height=10cm, width=14cm]{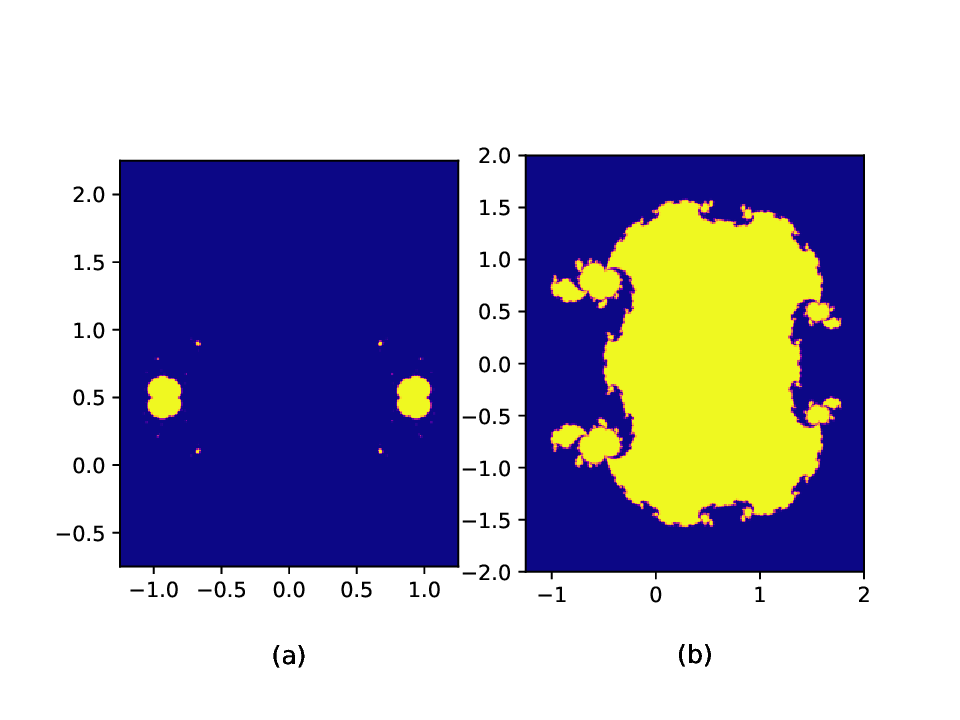}
    \caption{ $(a)$ The disconnected Julia set of $St_M(z)$ for $a=0$ and $(b)$ The Julia set $St_M(z)$ for $a\neq 0$ }
    \label{fig:3}
\end{figure}
\newpage
Now, we provide Table \ref{tab1} to compare the dynamical behaviours of Newton's and Stirling's iterative root-finding methods.
\begin{table}[h!]
    \centering
    \begin{tabular}{| m{3em}|  m{4cm} | m{4cm}|m{4cm}|}
        \hline
        \multicolumn{4}{|c|}{\textbf{Comparison Table}}\\
        \hline
        \textbf{SL. No.} & \textbf{Properties} & \textbf{Newton's method} & \textbf{Stirling's method}\\
         \hline
        1 & Order of the method & 2 & 2\\
        \hline
         2 & Scaling theorem & satisfies  & does not satisfy \\
        \hline
        \multicolumn{4}{|c|}{\textbf{For polynomials with simple zeroes}}\\
         \hline
         3 & Zeroes & zeroes are superattracting fixed points & zeroes are superattracting fixed points\\
         \hline         
         4 & $\infty$ & $\infty$ is repelling fixed point & $\infty$ is rationally indifferent fixed point\\
         \hline
         5 & Julia set & Connected \cite{shishikura1990connectivity} & Connected\\
         \hline
         \multicolumn{4}{|c|}{\textbf{For M\"{o}bius maps}}\\
         \hline
         6 & Zeroes & zeroes are superattracting fixed points & zeroes are superattracting fixed points\\
         \hline         
         7 & $\infty$ &  repelling fixed point &  attracting (superattracting) for $a=0$ $(a\neq 0)$\\
         \hline
         8 & Extraneous fixed points & exist and repelling \cite{buff2003konig} & rationally indifferent\\
         \hline
         9 & Invariant parabolic domains & does not exist & exist\\
         \hline
         10 & Invariant Siegel discs & does not exist & does not exist\\
         \hline
    \end{tabular}
    \caption{A comparison table between the dynamics of Newton's method and Stirling's iterative method}
    \label{tab1}
\end{table}
\newpage
\section*{Future scope}
In this paper, we investigate the dynamics of Stirling's iterative method for rational functions, specially for polynomials and M\"{o}bius maps. Our future aim is to study the dynamics of this method for polynomials with repeated zeroes and different families of functions. For the M\"{o}bius map, we show that it is possible to exist the Herman ring. In future, we will try to find out those M\"{o}bius maps for which Herman rings exist. Also, we will study about the connectivity of the Julia set of $St_M(z)$ for the case $a\neq 0$. 
\section*{Acknowledgement}
 Both authors sincerely acknowledge the financial support rendered by the National Board for Higher Mathematics, Department of Atomic Energy, Government of India sponsored project with Grant No. 02011/17/2022/NBHM (R.P)/R\&D II/9661 dated: 22.07.2022.

\section*{Statement and declaration}
No potential competing interest was reported by the authors.

\section*{Data availability statement}
Data sharing not applicable to this article as no datasets were generated or analysed during the current study.
\bibliographystyle{amsplain}
\bibliography{bibliography}

\providecommand{\bysame}{\leavevmode\hbox to3em{\hrulefill}\thinspace}
\providecommand{\MR}{\relax\ifhmode\unskip\space\fi MR }
\providecommand{\MRhref}[2]{%
  \href{http://www.ams.org/mathscinet-getitem?mr=#1}{#2}
}
\providecommand{\href}[2]{#2}
\begin{thebibliography}{10}

\bibitem{amat2004review}
S.~Amat, S.~Busquier, and S.~Plaza, \emph{Review of some iterative root-finding methods from a dynamical point of view}, Scientia \textbf{10} (2004), no.~3, 35.

\bibitem{amoros2019extending}
C.~Amor{\'o}s, I.~K Argyros, {\'A.}~A. Magre{\~n}{\'a}n, S.~Regmi, R.~Gonz{\'a}lez, and J.~A. Sicilia, \emph{{Extending the applicability of Stirling’s method}}, Mathematics \textbf{8} (2019), no.~1, 35.

\bibitem{argyros2018expanding}
I.~K Argyros and P.~K. Parida, \emph{{Expanding the Applicability of Stirling’s Method under Weaker Conditions and Restricted Convergence Regions}}, Annals of West University of Timisoara-Mathematics and Computer Science \textbf{56} (2018), no.~1, 86--98.

\bibitem{beardon2000iteration}
A.~F. Beardon, \emph{{Iteration of rational functions: Complex analytic dynamical systems}}, vol. 132, Springer Science \& Business Media, 2000.

\bibitem{buff2003konig}
X.~Buff and C.~Henriksen, \emph{{On K{\"o}nig's root-finding algorithms}}, Nonlinearity \textbf{16} (2003), no.~3, 989.

\bibitem{chakra2016baker}
T.~K. Chakra, G.~Chakraborty, and T.~Nayak, \emph{{Baker omitted value}}, Complex Variables and Elliptic Equations \textbf{61} (2016), no.~10, 1353--1361.

\bibitem{chakra2018herman}
\bysame, \emph{{Herman rings with small periods and omitted values}}, Acta Mathematica Scientia \textbf{38} (2018), no.~6, 1951--1965.

\bibitem{chakraborty2021configurations}
G.~Chakraborty, S.~K. Datta, and S.~Sahoo, \emph{{Configurations of Herman rings in the complex plane}}, Indian J. Math \textbf{63} (2021), no.~3, 375--391.

\bibitem{kneisl2001julia}
K.~Kneisl, \emph{{Julia sets for the super-Newton method, Cauchy’s method, and Halley’s method}}, Chaos: An Interdisciplinary Journal of Nonlinear Science \textbf{11} (2001), no.~2, 359--370.

\bibitem{milnor2011dynamics}
J.~Milnor, \emph{{Dynamics in one complex variable}}, vol. 160, Princeton University Press, 2011.

\bibitem{nayak2022quadratic}
T.~Nayak and S.~Pal, \emph{{Quadratic and cubic Newton maps of rational functions}}, Proceedings-Mathematical Sciences \textbf{132} (2022), no.~2, 46.

\bibitem{nayak2022julia}
\bysame, \emph{{The Julia sets of Chebyshev’s method with small degrees}}, Nonlinear Dynamics \textbf{110} (2022), no.~1, 803--819.

\bibitem{shishikura1987quasiconformal}
M.~Shishikura, \emph{On the quasiconformal surgery of rational functions}, Annales scientifiques de l'{\'E}cole Normale Sup{\'e}rieure, vol.~20, 1987, pp.~1--29.

\bibitem{shishikura1990connectivity}
\bysame, \emph{{The connectivity of the Julia set and fixed points}}, Preprint, Institut des Hautes Etudes Scientifiques IHES/M/90/37 (1990).

\bibitem{wang2007julia}
X.~Wang and X.~Yu, \emph{{Julia sets for the standard Newton’s method, Halley’s method, and Schr{\"o}der’s method}}, Applied Mathematics and computation \textbf{189} (2007), no.~2, 1186--1195.

\end{thebibliography}
\end{document}